\documentclass[11pt]{amsart}
\usepackage{hyperref}
\usepackage[usenames]{color}
\usepackage{amsmath,amsthm,amsfonts,amssymb}

\newtheorem{theorem}{Theorem}[section]
\newtheorem{lemma}[theorem]{Lemma}
\newtheorem{corollary}[theorem]{Corollary}
\newtheorem{proposition}[theorem]{Proposition}

\theoremstyle{definition}
\newtheorem{definition}[theorem]{Definition}

\theoremstyle{remark}

\numberwithin{equation}{section}

\def\CE{\mathcal E}
\def\CN{\mathcal N}

\def\Z{\mathbb Z}
\def\N{\mathbb N}

\newcommand{\inv}{^{-1}}
\newcommand{\abs}[1]{\vert #1 \vert }

\begin{document}

\title{Random equations in nilpotent groups}

\author{Robert H. Gilman}
\address{Department of Mathematical Sciences\\Stevens Institute of Technology}
\curraddr{}
\email{Robert.Gilman@stevens.edu}
\thanks{}

\author{Alexei Myasnikov}
\address{Department of Mathematical Sciences\\Stevens Institute of Technology}
\curraddr{}
\email{amiasnikov@gmail.com}
\thanks{The second author was partially supported by  NSF grant
DMS-0405105,  NSERC Discovery grant RGPIN 261898, and NSERC Canada Research Chair
 grant.}

\author{Vitali\u\i\ Roman'kov}
\address{Institute of Mathematics and Information Technologies\\Omsk State Dostoevskii University}
\curraddr{}
\email{romankov48@mail.ru}
\thanks{The third author was partially supported by RFBR Grant 10-01-00383. The third author  is very grateful to Stevens Institute of Technology for its hospitality.}

\subjclass[2000]{Primary 20P05; secondary 20F69, 20F10, 20F18, 16W20, 13B10, 13B25, 20F05.}

\keywords{free abelian groups, free nilpotent groups, free groups, equations, asymptotic density}

\date{}


\begin{abstract}

In this paper we study satisfiability of random equations in an infinite finitely generated nilpotent  group $G$.   Let $SAT(G,k)$ be the set of all satisfiable equations over $G$ in $k$ variables. For a free abelian group $A_m$ of rank $m$ we show that the ball asymptotic density $\rho(SAT(A_m,k))$ of the set $SAT(A_m,k)$ (in the whole space of all equations in $k$ variables over $G$) is equal to $0$ for $k = 1$, and is equal to $\zeta(k+m)/\zeta(k)$ for $k\ge 2$.  More generally, if $G$ is a finitely generated nilpotent infinite group then again the asymptotic density of the set $SAT(G,1)$ is $0$. For $k\geq 2$ we give robust estimates for upper and lower asymptotic densities of the set $SAT(G,k)$. Namely, we prove that 
these densities lie in the interval  from $ \frac{1}{t(G)}\frac{\zeta (k +  h(G))}{\zeta (k)}$ to $\frac{\zeta (k + m)}{\zeta (k)}$, where $h(G)$ is the Hirsch length of $G$, $t(G)$ is the order of lower central torsion of $G$, and  $m$ is the torsion-free rank (Hirsch length) of the abelianization of $G$. This allows one to describe the asymptotic behavior of the set $SAT(G,k)$ for different parameters $k,m,h(G),t(G)$.

\end{abstract}


\maketitle

\section{Introduction}
\label{se:intro}

In this paper we study satisfiability of random equations in an infinite finitely generated nilpotent  group $G$.  We show that the set $SAT(G,k)$ of all equations in $k>1$ variables over an infinite finitely generated nilpotent group which are satisfiable in $G$  has an intermediate asymptotic density  in the space of all equations in $k$ variables over $G$. When $G$ is free abelian group of finite rank  we compute this density precisely, otherwise we give some non-trivial upper and lower bounds. 
For $k = 1$ the set $SAT(G,k)$ is negligible. 

In finite group theory the idea of genericity can be traced to  works of Erd\H{o}s and Turan \cite{ET} and Dixon \cite{Di}. Nowadays, it is a well-established subject with powerful techniques and   beautiful results (for recent developments we refer to surveys in ~\cite{KL,LiSh1,Shalev3,DPSSh,Shalev4}), while in geometric  group theory the  generic approach is due to Gromov. His inspirational  works \cite{Gro1,Gro2,Gro3} turned the subject into an area of very active research, see for example~\cite{AO,A2,A4,BMR1,BMR2,BMR3,BogV,BV,BMS,CERT,BorMyas,BV,Ch1,Ch2, Jit, KSS,KRSS,KR,Ollivier,Olsh,Rom,MTV,Zuk}.   Quite often random walks  provide general  methods and techniques to study probabilistic problems in groups  (we  refer to the book \cite{Woess2} as a general reference on random walks on infinite groups).

Experience gained during the development of the Magnus computational group theory package~\cite{Magnus} motivated a novel  approach to algorithmic problems in groups, shifting the focus onto their generic complexity (complexity on most or ``generic'' inputs). The foundations of generic complexity in groups were laid down in a series of publications \cite{KMSS1,KMSS2,BMS,BMR1,GMMU,MSUbook}. 
It became clear that knowledge of  generic properties of groups, their elements, subgroups, automorphisms, etc,   can be used in design of  simple practical algorithms that work very fast on most inputs.  On explicit analysis of generic complexity of the Word, Conjugacy, and Membership  problems see \cite{KMSS1,BorMyas,BMR2,BMR3,FMR1,MU,MU2,GMO,MO1}.

 Recently, a host of papers appeared on generic properties of individual groups and related objects. 
We mention here,  in particular,  results  on generic properties of one-relator groups \cite{KS1,KS2};  random elements of mapping class groups and automorphisms of free groups \cite{maher,Riv2};  on Whitehead minimal elements \cite{KRSS} and SCL in free groups \cite{Cal1,Cal2}; random  van Kampen diagrams \cite{MU2} and  average Dehn functions \cite{BogV,KR,Rom}.  Another interesting development concerns with random subgroups of a given group. The following groups have been studied: random groups \cite{AO,A2,A3}, braid group \cite{MU,MSU}, hyperbolic group \cite{GMO}, Thompson's group $F$ \cite{CERT}, linear groups \cite{Aoun,Riv3}, and amalgams of groups \cite{FMR}. 

 In cryptography, several successful attacks have exploited  generic properties of randomly chosen  objects to break cryptosystems \cite{MU,MSUbook,MSU,RST,GK}. 

Beyond the Word and Conjugacy problems not much is known about generic properties of equations in groups - this paper is the first in this direction. At the end of the paper we formulate several open problems on random equations in groups.

\subsection{Asymptotic densities}
\label{subsec:densities}

A {\em stratification} of a countable set $T$ is a  sequence $\{T_m\}_{m \in \mathbb{N}}$ of non-empty finite subsets $T_m$ whose union is $T$. 
Stratifications are often specified by size functions. A {\em size} or {\em length function} on $T$ is a map $\sigma:T\to\N$ from $T$ to the nonnegative integers $\N$ such that the inverse image of every integer is finite. The corresponding  ball stratification is formed by {\em balls} $B_m = \{t \in T \mid \sigma(t) \leq m\}$. One can also  consider stratifications by {\em spheres}, but in this paper, if not said otherwise, all the stratifications $\{T_m\}$ are the ball ones relative to some chosen size functions. 

The {\em asymptotic density} of $A\subset T$ with respect to a stratification $\{T_m\}$ is defined to be
$$\rho(A) = \lim_{m\to\infty} \rho_m(A), \ \ \ where \ \ \  \rho_m(A) = \frac{\abs{A\cap T_m}} {\abs{T_m}},$$
 when the limit exists. Otherwise, we use the limits 
$$\bar{\rho}(A) = \limsup_{m\to\infty} \rho_m(A), \ \ \ \underline{\rho}(A) = \liminf_{m\to\infty} \rho_m(A).$$
Note, that some authors call $\rho(A)$ the strict asymptotic density and refer to $\bar{\rho}$ as the standard asymptotic density.

$A$ is said to be {\em generic} in $T$ with respect to the stratification $\{T_m\}_{m \in \mathbb{N}}$ if $\rho(A) = 1$ and {\em negligible} if $\rho(A) = 0$. A property of elements of $T$ is generic if it holds on a generic subset.

Similarly, one can define relative genericity of subsets in $T$. For subsets $A \subseteq  B \subseteq T$  we define first a relative frequency $\rho_m (A | B)$ of $A$ in $B$ with respect to $\{T_m\}_{m \in \mathbb{N}}$ as the  ratio
$$\rho_m (A | B) = \frac{\abs{A\cap T_m}}{\abs{B \cap T_m}}  = \frac{\rho_m(A)}{\rho_m(B)}.$$
 Here we treat the fractions $\frac{0}{0}$ as $0$. In particular,  $\rho_m(A | T) = \rho_m(A)$.  Now the asymptotic density of $A$ relative to $B$ is defined by
 $$
 \rho(A|B) = \lim_{m\to\infty} \rho_m(A | B),
 $$
 and 
 $$\bar{\rho}(A|B) = \limsup_{m\to\infty} \rho_m(A|B).$$
 We say that $A$ is generic in $B$ relative to $\{T_m\}_{m \in \mathbb{N}}$ if $\rho(A|B) = 1$ and negligible if the limit is $0$.

\subsection{Random equations over groups}
\label{subsec:equations}

For general facts on equations and algebraic geometry over groups we refer to \cite{BMR}. An equation $u = 1$  in $k$ variables over a group $G$ is an expression of the form $g_0x_{i_1}^{m_1}g_1\ldots x_{i_n}^{m_n}g_n = 1$ where each coefficient $g_j$ is a group element, each exponent $m_j$ is an integer, and each $x_{i_j}$ is taken from an alphabet of variables, $X=\{x_1,\ldots, x_k\}$. In this case, the free product $G_X = F(X)\ast G$ is  the space of all equations in variables $X$ and coefficients in $G$.  A solution of $u =1$ in $G$ is an assignment $x_j\to h_j\in G$ such that $g_0h_{i_1}^{m_1}g_1\ldots h_{i_n}^{m_n}g_n = 1$. Denote by $SAT(G,K)$ the set of all  equations from $G_X$ which have a solution in $G$ ({\em satisfiable} equations). We are interested in the following  question: what is the  probability of a {\em random equation} from $G_X$ to be satisfiable in $G$? First of all, the answer depends on the given distribution on $G_X$.  If $G$ is a free group then there are natural and easy to compute distributions on $G_X$ (see \cite{BMS,BMR2}), otherwise working with measures on arbitrary free products could be rather technical. We refer to a series of  papers \cite{BMR3,BMR2,FMR1} on asymptotic properties of the Conjugacy Problem in free products  with amalgamation and HNN extensions.     Satisfiability of random equations in free groups is studied in \cite{GMR2}. Solvability of random equations without variables, i.e,  random instances of the  Word Problem,  in a given finitely presented group $G$ is an important problem  with interesting applications. We refer to \cite{KMSS1} and \cite{MU} for recent progress in this area.

 Furthermore, besides the problem of choosing a proper distribution on $G_X$ there is another important issue - what is the proper space of equations for a given group $G$?  The free product $G_X$ seems suitable  if the group $G$ is free, or non-elementary hyperbolic, or any other group that generate the variety of all groups. Meanwhile,  if $G$ is an abelian group, then it would be natural to consider  an equation $abx_2b\inv x_1=1$ as essentially the same  as $ax_1x_2=1$. More generally, it is natural to  consider two equations as essentially the same if one can be transformed into the other by applying identities of  the variety of abelian groups. So the natural space of equations in variables $X$ over an abelian group $G$ is the direct product $G \times A(X)$ of $G$ and a free abelian group $A(X)$  with basis $X$. The unifying approach is that if $G$ is in a variety ${\mathcal V}$ then the space of equations in variables $X$ over $G$ {\em relative to ${\mathcal V}$}  is the free product in the variety ${\mathcal V}$ of $G$ and a  ${\mathcal V}$-free group with basis $X$. Following this approach for a nilpotent group $G$  of nilpotency class $c$ we define the space of {\em nilpotent} equations as the set of elements of a free product $ G \ast_{\mathcal N} F_{\mathcal N}(X)$, where ${\mathcal N}$ is the variety of all nilpotent groups of class $\leq c$, $\ast_{\mathcal N}$ is the free product in ${\mathcal N}$, and $F_{\mathcal N}(X)$ is a free nilpotent group of class $c$ with basis $X$ (see Section \ref{sec:nilpotent} for details). Moreover, the group $G_X = G \ast_{\mathcal N} F_{\mathcal N}(X)$ has a nice set of unique normal forms of elements that come from a fixed base of $G_X$ of a certain type (sometimes called Malcev or polycyclic bases). So the actual set of nilpotent equations for us is this set of normal forms of elements of $G_X$ with respect to a chosen base.  There is a robust ball stratification of $G_X$ relative to the norm $||\cdot||$ on $G_X$ which comes naturally with the base. Now, as is customary in discrete groups, we use the asymptotic density relative to the ball stratification  to measure various  subsets  of equations  in $G_X$, in particular the sets $SAT(G,k)$.

\subsection{Survey of results}
\label{subsec:survey}

In Section~\ref{sec:abelian} we compute the (ball) asymptotic density $\rho(SAT(A, k))$ of a free abelian group $A = \Z^m$  with respect to the size function $||\cdot||$, where $||(a_1,\ldots, a_m)|| = \max\{\abs{a_i}\}$. The main result of the section is Theorem~\ref{th:abelian-two-var}, which states that $\rho(SAT(A, k)) = \zeta(k+m)/\zeta(k)$ for $k\ge 2$.  For $k = 1$ it is not hard to prove (Corollary~\ref{co:estimates})that $\rho(SAT(A, 1)) = 0$. In particular, $SAT(\Z^m,k)$ has an intermediate asymptotic density for equations in more then one variable. This is not very surprising, though to get the actual number requires some work. The proof of Theorem~\ref{th:abelian-two-var} is based on explicit knowledge of the rates of convergence for the asymptotic densities of the set of $\gamma$-th powers in $\Z^m$ and of the set of $m$-tuples  with the greatest common divisor $\gamma$. The densities themselves are known (\cite{KRSS} and \cite{HW, Mert,Chr,KRSS}), but the convergence rates do not seem to be in the literature. We estimate them in Propositions~\ref{pr:gamma-Z-k} and~\ref{pr:2.5}.

In Section~\ref{sec:nilpotent} we extend the results of Section~\ref{sec:abelian} to  an arbitrary  finitely generated infinite  nilpotent group $G$ of class $c$.  In this case an equation $u = 1$ over $G$ is viewed as an element of the nilpotent group $G_X$. The size function on $ G_X$ is the norm $||\cdot||$, where $||u||$ for $u \in G_X$ is given by 
$$
||u|| = \max\{|\alpha_1|, \ldots, |\alpha_s|\}
$$
and  $\alpha_1, \ldots,\alpha_s \in \mathbb{Z}$ are the coordinates of $u$ in a fixed Malcev basis of $G_X$ (see Section \ref{sec:nilpotent} for details). The asymptotic density $\rho$ is the ball density  on $G_X$ relative to the size function $||\cdot||$. 
We show that $\rho(SAT(G,1)) = 0$ while $SAT(G,k))$  has an intermediate asymptotic density for $k\geq 2$. More precisely, we show that  
$$
\frac{\zeta (k + m)}{\zeta (k)}  \geq \limsup_{n \to \infty} \rho_n(SAT(G,k)),
$$
and 
$$ 
 \liminf_{n \to \infty} \rho_n(SAT(G,k))  \geq \frac{1}{t(G)}\frac{\zeta (k +  h(G))}{\zeta (k)},
 $$
  where $h(G)$ is the Hirsch length of $G$, $t(G)$ is the order of lower central torsion of $G$,
and  $m$ is the torsion-free rank (Hirsch length) of the abelianization $G/G_2$ of $G$ (see Section \ref{sec:nilpotent} for precise definitions). Thus, the set $SAT(G,k)$ has intermediate asymptotic densities for $k\geq 2$. Again, this is not very surprising,  since solvability of an equation $u = 1$ in a nilpotent group $G$  is often determined by its solvability in the abelianization of $G$. For example,  Shmelkin's theorem \cite{Shm} shows that $u = 1$ is solvable in $G$ if the greatest common divisor of the exponents of the variables of $u$ is equal to $1$. This contributes $\frac{1}{\zeta (k)}$ to the lower bound above. However, the complete picture is much more complex and this is reflected in the bounds above.  
These bounds allows one to get a qualitative description of the set $SAT(G,k)$ relative to various parameters involved. Thus, if $G$  is a free nilpotent group $N_{m,c}$of rank $m$ and class $c$ then the lower central torsion $t(G)$  of $G$  is equal to 1, so the asymptotic densities of $SAT(N_{m,c},k)$ for $k \geq 2$ lie in the interval 
$$\left [\frac{\zeta (k +  h(N_{m,c}))}{\zeta (k)}, \frac{\zeta (k + m)}{\zeta (k)} \right ].$$
Hence, for fixed $m, c$  when $k \to \infty$ the lower density of $SAT(N_{m,c},k)$ approaches 1, i.e., $SAT(N_{m,c},k)$ is "almost" generic for $k >>1$. On the other hand, if $k\geq 2$ and $c$ are fixed but $m \to \infty$ then the interval above converges to the point $\frac{1}{\zeta (k)}$, so $SAT(N_{m,c},k)$ when the rank $m$ increases behaves asymptotically similar to the case described in  Shmelkin's theorem.  Yet, in the third scenario when $k, m \geq 2$  are fixed but $c \to \infty$ (hence $h(G) \to \infty$)  the interval converges to $\left [\frac{1}{\zeta (k)}, \frac{\zeta (k + m)}{\zeta (k)} \right ]$, so the asymptotic behavior of  $SAT(N_{m,c},k)$ is  somewhat in between the one described in  Shmelkin's theorem and the abelian case.

\section{Equations in free abelian groups}
 \label{sec:abelian}

In this section we compute the asymptotic density of the set $SAT(\Z^m ,k)$ of  all $k$-variable equations satisfiable in $\Z^m$.

\subsection{Preliminaries}
\label{subsec:preliminary-abelian}

Denote by $\N ^+$ the set of all positive integers.   From now on we fix a free abelian
group $A =  A(a_1, ... , a_m) \simeq \Z^m$   with basis $\{a_1, ... , a_m\}$ and  also a set
of variables $X = \{x_1, ... , x_k\}$,  where $m,  k \in \N ^+$.

As we have seen in Section \ref{subsec:equations} the space $A_X$ of all equations in variables
$X$ over the group $A$ is equal  to $A(a_1, ... , a_m) \times A(x_1, ... , x_k)$ - a free
abelian group with  basis $Y = \{x_1, ... , x_k, a_1, ... , a_m\}.$ In this case every element $w
\in A_X$ can be uniquely written in the form

\begin{equation}
\label{eq:abelian}
 w = x_1^{\gamma _1} ... x_k^{\gamma_k}
a_1^{\alpha_1} ... a_m^{\alpha_m},
\end{equation}

\noindent where $\gamma_1, ... , \gamma_k, \alpha_1, ... , \alpha_m \in \Z.$ The  norm $||w ||$
of an element $w \in A_X$ in the form (\ref{eq:abelian})  is defined by
 \begin{equation}
 \label{eq:norm-abelian}
  ||w|| = max \{ |\gamma_1|, ... , |\gamma_k|, |\alpha_1|, ... ,
|\alpha_m|\}.
\end{equation}
This norm gives a size function $|| \cdot || : A_X \to \N $ with the balls
 $$B_r = \{ w \in A_X \mid  ||w|| \leq r\}.$$
We call
$\gamma = gcd (\gamma_1, ... , \gamma_k) $ an  {\it exponent} of $w$ and denote it by $\gamma =
exp (w).$ In the exceptional case $\gamma_1 = ... = \gamma_k = 0$ we put $exp (w) = 0.$

  For each $r, \gamma \in
\N ^+$  a $\gamma$-{\em slice} of the ball $B_r$ is defined as
 $$B_{r,\gamma} = \{w\in B_r  \mid  exp (w) = \gamma \}.$$

We begin with a simple lemma which gives a robust criterion of satisfiability of equations in free
abelian groups.

\begin{lemma}
\label{le:solutions-abelian} Let $w = x_1^{\gamma _1} ... x_k^{\gamma_k} a_1^{\alpha_1} ...
a_m^{\alpha_m} \in A_X$  and $\gamma = exp (w)$. Then the following holds:
 \begin{enumerate}
  \item [1)] If $\gamma \neq 0$ then $w = 1$ has a solution in $A$ if and only if
   $\gamma$ divides $gcd(\alpha _1, ... , \alpha_m).$
   \item [2)]  If $\gamma = 0$ then $w = 1$ has a solution in $A$ if and only if $\alpha_1 = ... =
\alpha_m = 0$.
 \end{enumerate}
  \end{lemma}

\begin{proof} To prove 1) assume that $\gamma = exp(w) \neq 0$. If an  equation $w = 1$ has a solution  in $A$
then $a = a_1^{\alpha_1} ... a_m^{\alpha_m}$ is a $\gamma$-power of some element in $A$, hence
$\gamma$ divides each exponent $\alpha_i, i = 1, \ldots, k$, hence $\gamma$
divides $gcd(\alpha _1, ... , \alpha_m)$, as required. Conversely, if $\gamma$ divides
$gcd(\alpha_1, ... , \alpha_m)$ then $a  = b^\gamma$ for some $b \in A$.
Since $\gamma = gcd(\gamma_1, \ldots, \gamma_k)$ there are integers
$\xi_1, \ldots, \xi_k$ such that $\xi_1\gamma_1 + \ldots + \xi_k\gamma_k = \gamma$. It is easy to check
that the map $x_1 \to b^{-\xi_1}, \ldots, x_k \to b^{-\xi_k}$ gives a solution of $w = 1$ in $A$.

The case 2) is  truly obvious.

 \end{proof}

Recall the standard algebraic notation
 $$\gamma \Z^m =
\{(\alpha_1, \ldots, \alpha_m)\in  \Z^m \mid \gamma
  \mbox{ divides} \ gcd(\alpha_1, \ldots, \alpha_m)\}.$$

\begin{corollary}
\label{co:SAT(A-1)}
 $ SAT(A,1) = \{x^{\pm\gamma }a_1^{\alpha _1} \ldots a_m^{\alpha _m} \mid \gamma \in \N ^+, (\alpha _1,
 \ldots , \alpha _m) \in
  \gamma\Z^m\} \cup \{1\}$.
\end{corollary}

To use the criterion from Lemma \ref{le:solutions-abelian} we need several known and some perhaps unknown facts from elementary number theory.

   Recall, that an  element $x = x_1^{\gamma_1} ... x_k^{\gamma _k}$ of
   a  free abelian group
  $A(X)$  is called {\it primitive}
 if it is a member of some basis of $A(X),$  or, equivalently,
   if $gcd (\gamma_1, ..., \gamma_k) = 1.$   A non-trivial element
 $x = x_1^{\gamma_1} ... x_k^{\gamma _k}$ is called  $\gamma${\it
 -primitive} (here $\gamma \in \N ^+$)
  if and only if $gcd (\gamma _1, ... , \gamma_k) = \gamma
 $, or, in other words,  $x$ is a $\gamma $-power of some primitive
element in $A(X).$ For $\gamma, k \in \N ^+$ denote by $P_{k,\gamma }$ the set of all
$\gamma $-primitive elements in the free abelian group $A(X)$ of rank $k$.

 The following result is well-known in number theory; in full generality it is due
 to Christopher~\cite{Chr}, see also~\cite{Mert, HW}
 for earlier partial results.  Below,   $\zeta(t) = \sum_{n = 1}^{\infty} \frac{1}{n^t}$
 is Riemann zeta-function.

\begin{proposition}\cite{Chr}
\label{pr:Chris} Let $\gamma , k \in N^+, k \geq 2.$ Then
   $$\rho (P_{k,\gamma}) = \frac{1}{\gamma^k \zeta (k)}.$$
\end{proposition}

The first part of the following result is an easy corollary
(see  the argument in~\cite{KRSS}) of Proposition \ref{pr:Chris}.
The two other  parts concern with estimates on the convergence rates
that we could not find in the literature, so we give a short proof of both.

\begin{proposition}
\label{pr:gamma-Z-k}
 Let $\gamma, k \in \N ^+.$ Then:
 \begin{itemize}
 \item [1)] $\rho (\gamma \Z^k) = 1/\gamma ^k$;
  \item [2)]
 $|\rho_r(\gamma \Z^k) - \frac{1}{\gamma ^k}| \leq \frac{2^{k+1}k}{r\gamma ^{k-1}}$
  for every $r \geq \gamma$.
  \item [3)] $\rho_r(\gamma \Z^k)$ converges to $1/\gamma ^k$ uniformly in $\gamma$.
 \end{itemize}
\end{proposition}
\begin{proof} Notice that $|\gamma \Z^k \cap B_r| = (2 \lfloor r/\gamma \rfloor  + 1)^k $, where  $\lfloor r/\gamma \rfloor$ is the integral part of $r/\gamma$.
Put $t = [r/\gamma ] \in N.$ Then
$$\rho_r(\gamma \Z^k) = \rho_{\gamma t}(\gamma \Z^k) = \frac{(2t + 1)^k}{(2r + 1)^k}.$$
 Clearly, the function $ r \rightarrow \rho_r(\gamma \Z^k)$ is decreasing on the segment  $[\gamma t,\gamma (t+1)- 1]$. Denote
$$ x = \frac{2t + 1}{2\gamma (t+1) - 1}, \ \ \ y = \frac{2t+1}{2\gamma t + 1}.$$
  It is easy to check that $x \leq  1/\gamma  \leq y ,$ hence
 \begin{equation}
 \label{eq:d-gamma-1}
  \rho_{\gamma (t+1) - 1}(\gamma \Z^k)   \leq \frac{1}{\gamma^k } \leq \rho_{\gamma t} (\gamma \Z^k).
\end{equation}
To estimate $|\rho_r(\gamma \Z^k) - 1/\gamma ^k|$ consider
\begin{equation}
 \label{eq:d-gamma-2}
\rho_{\gamma t} (\gamma \Z^k) - \rho_{\gamma (t+1) - 1} (\gamma
\Z^k) =  y^k - x^k = ( y - x)(y^{k-1} + y^{k-2}x + ... + x^{k-1}).
\end{equation}
Observe  that
 \begin{equation}
 \label{eq:d-gamma-3}
 0 < x \leq y  \leq \frac{2}{\gamma } \ \ and \ \    0 \leq y - x  \leq \frac{4}{r}
 \end{equation}
  for every  $r \in  [\gamma t, 2\gamma (t+1) - 1].$ Now it follows from  (\ref{eq:d-gamma-1}),
  (\ref{eq:d-gamma-2}), and (\ref{eq:d-gamma-3}) that
$$|\rho_r(\gamma \Z^k) - \frac{1}{\gamma ^k}| \leq \rho_{\gamma t }(\gamma \Z^k) -
\rho_{\gamma (t+1)-1}(\gamma \Z^k) = $$
$$( y - x)(y^{k-1} + y^{k-2}x + ... + x^{k-1})
\leq \frac{2^{k+1}k}{r\gamma ^{k-1}}, $$
as claimed. This proves 2).

To see 3) notice that from 2) for every $r \geq \gamma$  (since $\gamma \geq 1$) one has
$$|\rho_r(\gamma \Z^k) - \frac{1}{\gamma ^k}| \leq \frac{2^{k+1}k}{r}. $$
 On the other hand, for $r < \gamma$
$$|\rho_r(\gamma \Z^k) - \frac{1}{\gamma ^k}|  = |\frac{1}{(2r+1)^k} - \frac{1}{\gamma ^k}| \leq \frac{2}{r^k},$$
 and  3) follows.
 \end{proof}

 The first part of the  following proposition is well-known in number theory
 (in the full generality it is due to Christopher~\cite{Chr}, see also~\cite{Mert, HW}
 for earlier partial results).  Since we need estimates on the convergence rates we prove the result
 and provide the estimates.

\begin{proposition}
\label{pr:2.5} Let $\gamma , k \in N^+, k \geq 2.$ Then
  \begin{itemize}
 \item [1)] For every $\varepsilon > 0$ there
exists $r(\varepsilon ) \in N^+$ such that
$$|\rho_r(P_{k,\gamma }) - \frac{1}{\gamma ^k
\zeta (k)}| \leq \frac{\varepsilon}{\gamma^{k-1}} $$
\noindent for every $r \geq r(\varepsilon )$ and any $ \gamma \leq r.$
\item [2)]
  $\rho_r(P_{k,\gamma})$ converges to $\frac{1}{\gamma^k \zeta (k)}$ uniformly in $\gamma$.

\end{itemize}
\end{proposition}

\begin{proof}
Let $\gamma , k \in N^+, k \geq 2$, and $r \geq \gamma$. Observe, first, that
$$P_{k, \gamma } \cap B_r = \gamma (P_{k, 1} \cap B_t)$$
\noindent for every $r \in [\gamma t , \gamma (t+1) - 1],$ where $t = \lfloor r/\gamma \rfloor.$

By Proposition \ref{pr:Chris} for any $\varepsilon > 0$ there exists  $s = s(\varepsilon ) \geq 1$ such that for every $r \geq  s $
$$ |\rho_r(P_{k, 1}) - \frac{1}{\zeta (k)}| \leq \varepsilon /3.$$
We fix an arbitrary $\varepsilon$ such that $1 > \varepsilon > 0$.
Let
$$\delta = \delta (s) = max_{r \leq s}\{|\rho_r(P_{k, 1}) - \frac{1}{\zeta (k)}|\}.$$
Then there exists $\mu $ such that
$$\frac{\delta }{\mu } \leq \frac{\varepsilon}{3}, \ \ \ \frac{2^{k+1}k}{\mu } \leq
\frac{\varepsilon}{3}.$$
 We claim that for every $r \geq s \mu $ and for every $\gamma \leq r$
$$
|\rho_r(P_{k, \gamma }) - \frac{1}{\gamma ^k \zeta (k)}| \leq \frac{\varepsilon}{\gamma^{k-1}}.
$$
Indeed, observe that
$$
|\rho_r(P_{k, \gamma }) - \frac{1}{\gamma ^k \zeta (k)} | = \left |\frac{|P_{k,1} \cap B_{t}|}{|B_r|} - \frac{1}
{{\gamma^k} \zeta (k)}   \right  |=
$$
$$
\left |\frac{|P_{k,1} \cap B_{t}|}{|B_{t}|}
\cdot \frac{|B_{t}|}{|B_r|} - \frac{1}{\gamma ^k \zeta (k)}  \right |
$$
and introducing notation
$$
\varepsilon (t) =  \frac{|P_{k,1} \cap B_{t}|}{|B_{t}|} - \frac{1}{\zeta (k)}, \ \
\eta (r, \gamma ) = \frac{|B_{t}|}{|B_r|} -  \frac{1}{\gamma ^k}
$$
one can continue the sequence of the  equalities above
$$= \left |\left(\frac{1}{\zeta (k)} + \varepsilon (t)\right) \cdot \left (\frac{1}{\gamma ^k} +
\eta (r, \gamma ) \right ) - \frac{1}{\gamma ^k \zeta (k)} \right | =$$

$$\left |\frac{\varepsilon (t)}{\gamma ^k} + \varepsilon
(t) \eta (r, \gamma ) + \frac{\eta (r, \gamma )}{\zeta (k)} \right | \leq $$

$$ |\frac{\varepsilon (t)}{\gamma ^k}| + |\varepsilon
(t) \eta (r, \gamma )| + |\frac{\eta (r, \gamma )}{\zeta (k)} |.$$

By Proposition \ref{pr:gamma-Z-k}
\begin{equation}
\label{eq:eta}
 |\eta (r, \gamma )| \leq \frac{2^{k+1}k}{r \gamma ^{k-1}} \leq
\frac{2^{k+1}k}{\mu s \gamma ^{k-1}} \leq \frac{\varepsilon }{3 s\gamma ^{k-1}} \leq \frac{\varepsilon }{3 \gamma ^{k-1}}.
\end{equation}

 Assume first, that $t \geq s.$ Then by our assumption
$$|\varepsilon (t)| \leq \frac{\varepsilon}{3},
\frac{|\varepsilon (t)|}{\gamma ^k} \leq \frac{\varepsilon }{3
\gamma ^k}.$$
 \noindent Therefore
$$ |\varepsilon (t) \eta (r, \gamma )| \leq
\frac{\varepsilon }{3 \gamma ^{k-1}}, \ \ \ \ |\frac{\eta (r, \gamma )}{\zeta (k)}| \leq \frac{\varepsilon
}{3 \gamma ^{k-1} \zeta(k)} \leq \frac{\varepsilon
}{3 \gamma ^{k-1}}.$$
It follows that in this case
\begin{equation}
\label{eq:main-claim}
|\rho_r(P_{k, \gamma }) - \frac{1}{\gamma ^k \zeta (k)}| \leq \frac
{\varepsilon }{\gamma ^{k-1}}.
\end{equation}
Now suppose that $t =\lfloor r/\gamma \rfloor \leq s$.  Since $r \geq s \mu$ one has
$\gamma \geq \mu $.  Therefore
\begin{equation}
\label{eq:epsilon}
|\varepsilon (t)| \leq \delta , \ \ \
|\frac{\varepsilon (t)}{\gamma ^k}| \leq \frac{\delta }{\mu \gamma ^{k-1}} \leq
\frac{\varepsilon }{3 \gamma ^{k-1}}.
\end{equation}

Now from (\ref{eq:eta}) and (\ref{eq:epsilon}) we again have (\ref{eq:main-claim}), as claimed. This proves 1).

To see 2) observe first that for $r < \gamma$ one has
 $$
 |\rho_r(P_{k,\gamma }) - \frac{1}{\gamma ^k
\zeta (k)}| = |\frac{1}{(2r+1)^k} - \frac{1}{\gamma ^k
\zeta (k)}| \leq \frac{1}{(2r+1)^k} + \frac{1}{r^k
\zeta (k)}.
 $$
Combining this with 1) one  gets the required uniform convergence.

\end{proof}

\subsection{Asymptotic properties of one-variable equations}

  Recall, that $SAT(A,1)$ is the set of all
  equations in a single variable $x$ which are satisfiable in $A =
  A(a_1, \ldots,a_m)$. In computations below we use notation
  $$\Z_n(m) = \sum_{\gamma = 1}^n \frac{1}{\gamma^m}.$$
 Clearly, if $m \geq 2$ then $\Z_n(m) \rightarrow \zeta(m)$ when $n \rightarrow
 \infty$,  and $\Z_n(1) \sim \ln n$.

\begin{theorem}
\label{th:abelian-one-var} For $r, m \in \N ^+$
$$\mid \rho_r(SAT(A,1)) - \frac{\Z_r(m)}{r} \mid = O \left (\frac{ \Z_r(m-1)}{r^2} \right ).$$

 \end{theorem}
  \begin{proof}
One-variable equations over $A$ can be written in the form $x^\gamma a = 1,$ where $a \in A$,
$\gamma \in \Z$. Let $\pi:A_X \to A$ be the canonical projection, so $\pi(x^\gamma a) =
a$.

Let $r \in \N ^+$ and $B_r$  be the ball of radius $r$ in $A_X$. Our goal is to estimate the ratio
 $$\rho_r(Sat(A,1)) = \frac{|B_r \cap Sat(A,1)|}{|B_r|} = \frac{|B_r \cap Sat(A,1)|}{(2r+1)^{m+1}}.$$
   To estimate $|B_r \cap Sat(A,1)|$ we
  partition  $B_r$ as  a disjoint union of its semi-slices $C_{r, \gamma } = \{x^\gamma a  \mid ||x^\gamma a|| \leq r, a \in A\}$:
  $$B_r = \cup_{\gamma = -r}^r C_{r,\gamma},$$
 \noindent
 and estimate the
 input of each semi-slice. Note that for $\gamma \geq 0$ the slice $B_{r, \gamma }$ is union of its semi-slices $B_{r, \gamma } = C_{r, \gamma } \cup C_{r, -\gamma }.$

 Notice, that semi-slices $C_{r,\gamma}$ and $C_{r,\delta}$ do not
 intersect for $\gamma \neq \delta$, and for any $\gamma \in [-r,r]$
  \begin{equation}
  \label{eq:B-r-gamma-slice}
 |C_{r,\gamma}| = (2r+1)^m.
 \end{equation}

 Therefore,
\begin{equation} \label{eq:reduction-to-slices} \rho_r(Sat(A,1)) = \frac{|B_r \cap Sat(A,1)|}{|B_r|}
= \sum_{\gamma = -r}^{r}\frac{|C_{r,\gamma} \cap Sat(A,1)|}{(2r + 1)^{m+1}}.
 \end{equation}

From Corollary \ref{co:SAT(A-1)} one can see that
$$
\pi(C_{r,\gamma} \cap Sat(A,1)) = \pi(B_r) \cap \gamma\Z^m,
$$
hence
 $$
|C_{r,\gamma} \cap Sat(A,1)| = |\pi(C_{r,\gamma} \cap Sat(A,1))| = |\pi(B_r) \cap \gamma\Z^m|.
 $$
It follows that
$$
\frac{|C_{r,\gamma} \cap Sat(A,1)|}{(2r+1)^m} =  \rho_r(\gamma \Z^m),
$$
so from (\ref{eq:reduction-to-slices})
 \begin{equation}
 \label{eq:reduction-to-numbers}
\rho_r(Sat(A,1)) = \frac{1}{2r+1}\sum_{\gamma = -r}^{r}\rho_r(\gamma \Z^m).
\end{equation}
  Observe, that
  $$\rho_r(\gamma \Z^m) = \rho_r(-\gamma
  \Z^m), \   \ \rho_r(0 \cdot \Z^m) =
  \frac{1}{(2r+1)^m},$$
so
\begin{equation}
\label{eq:almost-precise} \rho_r(Sat(A,1)) = \frac{2}{2r+1}\sum_{\gamma = 1}^{r}\rho_r(\gamma
\Z^m) + \frac{1}{(2r+1)^{m+1}}.
 \end{equation}
By Proposition \ref{pr:gamma-Z-k} for $\gamma \in [1,r]$ one has
 $$
 \mid \rho_r(\gamma \Z^m) - \frac{1}{\gamma^m}\mid \leq
 \frac{2^{m+1}m}{r\gamma ^{m-1}}.
$$
This implies that
 $$
  \sum_{\gamma =
1}^{r}(\frac{1}{\gamma^m} - \frac{2^{m+1}m}{r\gamma ^{m-1}})  \leq \sum_{\gamma = 1}^{r}\rho_r(\gamma
\Z^m) \leq \sum_{\gamma = 1}^{r}(\frac{1}{\gamma^m} + \frac{2^{m+1}m}{r\gamma ^{m-1}}),
$$
hence
 $$
  \Z_r(m) - \frac{2^{m+1}m}{r}\Z_r(m-1)  \leq \sum_{\gamma = 1}^{r}\rho_r(\gamma
\Z^m) \leq \Z_r(m) + \frac{2^{m+1}m}{r}\Z_r(m-1).
$$
Combining this with (\ref{eq:almost-precise}) one gets after standard  manipulations
$$\mid \rho_r(SAT(A,1)) - \frac{\Z_r(m)}{r} \mid =
O(\frac{ \Z_r(m-1)}{r^2}).$$
 This proves the theorem.
\end{proof}

\begin{corollary}
\label{co:estimates} The set $Sat(A,1)$ is  negligible. Moreover,
$$
 \rho_r(SAT(A,1)) =
\left\{
\begin{array}{ll}
O(\frac{ln r}{r})  & \mbox{if } m = 1, \\
O(\frac{1}{r}) & \mbox{if } m \geq  2.
\end{array}
\right.
$$

\end{corollary}
  \begin{proof}
  If $m = 1$ then $\Z_r(m) \sim \ln
 r$, and $\Z_r(m-1) \sim r$,
 and the result follows.

  If $m \geq 2$ then $\Z_r(m) \rightarrow \zeta(m) <
 \infty$, and the result follows, as well.

  \end{proof}

\subsection{Asymptotic properties of multi-variable  equations}

 The results are different for equations with two or more variables.
 Let, as above, $A = A(a_1, \ldots,a_m)$, $X = \{x_1, \ldots,
 x_k\}$, and $A_X = A \times A(X)$. Denote by $\pi_A$ and $\pi_X$
 the projections of $A_X$ onto the corresponding  factors.
  Formally, the balls $B_r$, as well as the sets
 $P_{k,\gamma}$ and $\gamma \Z^k$, are defined only for abelian
 groups of the type $\Z^k$.  However, if $G$ is a free
 abelian group with a fixed basis $Y$, then we extend this notions
 to $G$ by identifying $G$ with $\Z^{|Y|}$ via the canonical representation
 of elements of $G$ in the basis $Y$.  To indicate which basis is
 used for identification we add a superscript $Y$ to the notation:
 $B_r^{(Y)}$, $P_{k,\gamma}^{(Y)}$, $\gamma \Z^{|Y|}$.

 \begin{theorem}
 \label{th:abelian-two-var}
 Let $k \geq 2.$ Then the strict asymptotic density of the set $Sat(A,k)$ exists and
 \begin{equation}
 \label{eq:Sat-A-k}
  \rho (SAT(A,k)) = \frac{ \zeta (k+m)}{\zeta (k)}.
  \end{equation}
  \end{theorem}

 \begin{proof}
Let $B_r$ be the ball of radius $r$ in $A_X$. Our goal is to estimate the ratio
 $$\frac{\mid SAT(A,k) \cap B_r\mid}{\mid B_r\mid} = \frac{\mid SAT(A,k) \cap B_r\mid}{(2r+1)^{k+m}}.$$
 To estimate $\mid SAT(A,k) \cap B_r\mid$ we partition
$B_r$ into slices
$$
B_r = \cup _{\gamma = 0}^{r}B_{r,\gamma},
$$
where
$$
B_{r,\gamma} = \{w \in B_r \mid exp(w) = \gamma\}.
$$
Directly from the definition we have
$$
B_{r,\gamma} = P_{r,\gamma}^{(X)} \times B_r^{(A)},
$$
where $P_{r,\gamma}^{(X)}$ is the set of all $\gamma$-primitive elements in $B_r^{(X)}$, and
$B_r^{(A)}$ is the ball of radius $r$ in $A$. By Lemma \ref{le:solutions-abelian}
$$
SAT(A,k) \cap B_{r,\gamma} = P_{r,\gamma}^{(X)} \times (\gamma\Z^{\mid A \mid} \cap
B_r^{(A)}).
$$
If $\gamma = 0$ then $SAT(A,k) \cap B_{r,\gamma} $ contains only the trivial element of $A_X$.

\noindent For $1 \leq \gamma \leq r$ put
$$
\frac{\mid P_{r,\gamma}^{(X)} \mid}{(2r+1)^k} = \frac{1}{\gamma^k\zeta(k)} +
\varepsilon(r,\gamma),
$$
and
$$
\frac{\mid \gamma\Z^{\mid A \mid} \cap B_r^{(A)}) \mid}{(2r+1)^m}  = \frac{1}{\gamma^m} +
\delta(r,\gamma),
$$
Then
\begin{equation}\label{eq:SAT(A,k)}
\frac{\mid SAT(A,k) \cap B_r\mid}{\mid B_r\mid} = \frac{\sum_{\gamma = 0}^r\mid P_{r,\gamma}^{(X)}
\times (\gamma\Z^{\mid A \mid} \cap B_r^{(A)})\mid}  {(2r+1)^{k+m}}
\end{equation}
$$
= \frac{1 + \sum_{\gamma = 1}^r\mid P_{r,\gamma}^{(X)}\mid \cdot \mid \gamma\Z^{\mid A
\mid} \cap B_r^{(A)}\mid}{(2r+1)^{k+m}}
$$
$$
= \frac{1}{(2r+1)^{k+m}} + \sum_{\gamma = 1}^r (\frac{1}{\gamma^k\zeta(k)} +
\varepsilon(r,\gamma)) (\frac{1}{\gamma^m} + \delta(r,\gamma))
$$
$$
= \frac{1}{(2r+1)^{k+m}} + \frac{1}{\zeta(k)}\sum_{\gamma = 1}^r \frac{1}{\gamma^{k+m}} +
\frac{1}{\zeta(k)}\sum_{\gamma = 1}^r \frac{1}{\gamma^{k}} \delta(r,\gamma) +  $$
$$
\sum_{\gamma = 1}^r \varepsilon(r,\gamma) \frac{1}{\gamma^m} + \sum_{\gamma = 1}^r
\varepsilon(r,\gamma)\delta(r,\gamma).
$$
The second summand above converges to $\frac{\zeta(k+m)}{\zeta(k)}$ and  all others go to $0$ when
$r \rightarrow \infty$. Indeed, by Proposition \ref{pr:2.5} for every $\varepsilon > 0$  there exists $r(\varepsilon)$ such that  for any $r \geq r(\varepsilon)$
$$
 |\varepsilon(r,\gamma)|   \leq \frac{\varepsilon}{\gamma^{k-1}},
$$
 whereas, by Proposition \ref{pr:gamma-Z-k}
$$
\mid \delta(r,\gamma) \mid \leq \frac{2^{m+1}m}{r\gamma ^{m-1}}.
$$
Therefore,
$$ \frac{1}{\zeta(k)}\sum_{\gamma = 1}^r \frac{1}{\gamma^{k}} \delta(r,\gamma) +
\sum_{\gamma = 1}^r \varepsilon(r,\gamma) \frac{1}{\gamma^m} + \sum_{\gamma = 1}^r
\varepsilon(r,\gamma)\delta(r,\gamma) \leq
$$
$$ \frac{1}{\zeta(k)}\sum_{\gamma = 1}^r \frac{1}{\gamma^{k}} \frac{2^{m+1}m}{r\gamma ^{m-1}} +
\sum_{\gamma = 1}^r \frac{\varepsilon}{\gamma^{k-1}} \frac{1}{\gamma^m} + \sum_{\gamma = 1}^r
\frac{\varepsilon}{\gamma^{k-1}} \frac{2^{m+1}m}{r\gamma ^{m-1}} \leq
$$
$$ \frac{1}{\zeta(k)}\sum_{\gamma = 1}^r \frac{1}{\gamma^{k}} \frac{2^{m+1}m}{r\gamma ^{m-1}} +
\sum_{\gamma = 1}^r \frac{\varepsilon}{\gamma^{k-1}} \frac{1}{\gamma^m} + \sum_{\gamma = 1}^r
\frac{\varepsilon}{\gamma^{k-1}} \frac{2^{m+1}m}{r\gamma ^{m-1}} =
$$
$$  \frac{2^{m+1}m}{\zeta(k) r}\Z_r(k+m-1)  +
 \varepsilon \Z_r(k+m-1)  +  \frac{\varepsilon 2^{m+1}m}{r} \Z_r(k+m-2)
$$
and the result follows since $k \geq 2$ and $m \geq 1$.
This proves the theorem.

\end{proof}

\section{Equations in free nilpotent groups}
 \label{sec:nilpotent}

Let ${\mathcal N} = {\mathcal N}_c$ be the variety of all nilpotent groups of class at most $c$.  $\CN$ is defined by the law $\gamma_c(x_1, \ldots,x_{c+1}) = 1$, where  $\gamma_{c+1} = [x_1, \ldots,x_{c+1}] $ is the left-normed $c+1$-fold commutator defined inductively as $\gamma_{c+1} = [\gamma_c,x_{c+1}]$, $\gamma_1 = x_1$. For a group $H$ by $\gamma_i(H)$ we denote the verbal subgroup in $H$  generated by values of the word $\gamma_i$ in $H$ and by $Z(H)$ - the center of $H$.

 In a nilpotent group $G$ of class $c$ the  subgroups $G_i = \gamma_i(G)$ form the so-called  {\em lower central series}:
\begin{equation} \label{eq:lower-central}
G = G_1 > \ldots > G_c > G_{c+1} = 1.
\end{equation}

We fix notation for the following groups:  $N_c(X) = F(X)/\gamma_{c+1}(F(X))$ - free nilpotent of class $c$ with basis $X = \{x_1, ...,x_k\} \ (k \geq 1)$; $G$ - a  finitely generated   nilpotent group of class $c$;  
$G_X$ - a free product $G \ast_{\mathcal N} N_c(X)$   of $G$ and $N_c(X)$ in the variety ${\mathcal N}$, i.e., $G_X = (G \ast  N_c(X))/\gamma_{c+1}(G \ast  N_c(X))$. 

Notice, that the standard epimorphism $ G \ast  N_c(X) \to G \times N_c(X)$ factors through the quotient $(G \ast  N_c(X))/\gamma_{c+1}(G \ast  N_c(X))$, so the natural homomorphisms $G, N_c(X) \to G \ast_{\mathcal N} N_c(X)$ are embeddings and the canonical epimorphism $G_X \to G \times N_c(X)$ is identical on the subgroups $G, N_c(X)$.

As was mentioned in the introduction we view $G_X$ as the space of all equations in variables $x_1, \ldots, x_k$ with coefficients in $G$.

\subsection{ Malcev bases}
\label{se:Malcev}

In this section we introduce notation and recall some facts about bases and polycyclic presentations  in finitely generated nilpotent groups. The bases were initially introduced by Malcev in \cite{Mal1} and thoroughly studied in \cite{Hall}.  For  a modern exposition we refer to the books \cite{KargMerz,HoltBook,Sims}.  

Let $G$ be a group, $a =(a_1, \ldots,a_n)$  an $n$-tuple of elements in $G$, and $\alpha = (\alpha_1, \ldots, \alpha_n)$ an $n$-tuple of integers. By $a^{\alpha}$ we denote the following product
$$
a^{\alpha} = a_1^{\alpha_1}a_2^{\alpha_2} \ldots a_n^{\alpha_n}.
$$
Concatenation of two tuples $a$ and $b$ is denoted by  $ab$. A 1-tuple $(x)$ is usually identified with $x$.

Recall  that every finitely generated abelian group $A$ is a direct sum of cyclic groups:
$$
A = \langle a_1\rangle \times \ldots \langle a_s\rangle \times \langle b_1\rangle \times \ldots \langle b_t\rangle
$$
where $\langle a_i\rangle$ is an  infinite cyclic, and $\langle b_i\rangle$ is a finite cyclic of order $\omega_i = \omega(b_i)$.

Every element $g \in A$ can uniquely  be represented in the form
\begin{equation} \label{eq:base}
g = a_1^{\alpha_1} \ldots a_s^{\alpha_s}b_1^{\beta_1} \ldots b_t^{\beta_t}
\end{equation}
where $\alpha_i \in \mathbb{Z}$ and $\beta_j \in \{0,1,\ldots, \omega_j-1\}$. We call the tuple 
$$
a = (a_1, \ldots, a_s,b_1,\ldots,b_t)
$$ 
a {\em base} of $A$ and the tuple $\sigma(g) = (\alpha_1, \ldots,\alpha_s,\beta_1, \ldots,\beta_t)$ the  coordinate of $g$ in the base $a$. In this notation we write the equality (\ref{eq:base})  as follows $g = a^{\sigma(g)}$.

One can generalize the notion of  base to polycyclic groups.  Recall that a group $G$ is {\em polycyclic} if there  is a sequence of elements $a_1, \ldots, a_n \in G$ that generates $G$ such that if $G_i$ denotes the subgroup $\langle a_i, \ldots, a_n \rangle$   then  $G_{i+1}$ is normal in $G_i$ for every $i$. In this case 
\begin{equation}\label{eq:polycyclic-series}
G = G_1 \geq G_2 \geq \ldots \geq G_n \geq G_{n+1} = 1
\end{equation}
is termed a {\em polycyclic series} of $G$ and the tuple  $a=(a_1, \ldots, a_n)$ is a {\em base} of $G$.  Denote by $\omega_i$ the order of the group $G_i/G_{i+1}$, here 
$\omega_i = \infty$ if the order is infinite. We refer to the tuple  $\omega(a) = (\omega_1, \ldots, \omega_n)$ as the {\em order } of $a$. Now set  $Z_{\omega_i} = \mathbb{Z}$ if $\omega_i = \infty$ and $Z_{\omega_i} = \{0,1,2, \ldots, \omega_i -1\}$ otherwise. The number $h(G)$ of indices $i$ such that $\omega_i = \infty$ is independent of the choice of the base of $G$, it is called the {\it Hirsch length}  of $G$.

The following result easily follows by  induction on $n$.

\begin{lemma}
Let $a = (a_1, \ldots,a_n)$ be a base of a polycyclic group $G$ of order $\omega(a) = (\omega_1, \ldots, \omega_n)$ . Then for every $g \in G$ there is a unique decomposition of the following form:
\begin{equation}\label{eq:coordinates}
g = a_1^{\alpha_1} a_2^{\alpha_2} \ldots a_n^{\alpha_n} ,  \ \ \ \alpha_i \in Z_{\omega_i}.
\end{equation}
\end{lemma}
In the notation above for an element $g \in G$ the tuple $\sigma(g) = (\alpha_1, \ldots, \alpha_n)$  from (\ref{eq:coordinates}) is called the tuple of {\em coordinates } of $g$ in the base $a$. Sometimes we write the equality (\ref{eq:coordinates}) as $g = a^{\sigma(g)}$.

Now let  $G$ be a  finitely generated nilpotent group of nilpotency class $c$ (hence polycyclic).   Put $G_i = \gamma_i(G)$. Then  
$$
G = G_1 >  G_2 > \ldots > G_c > G_{c+1} = 1.
$$  
is the lower central series of $G$.  All the quotients $G_i/G_{i+1}$ are finitely generated abelian groups. Let $d_i$ be a tuple of elements from $G_i$ such that its image in $G_i/G_{i+1}$  is a base of  $G_i/G_{i+1}$.  Then the  tuple   $a = d_1d_2\ldots d_c$, obtained by concatenation from the tuples $d_1, \ldots, d_c$, is  a base of $G$, called a {\em lower central  base} of $G$.  Let $\omega$ be the order of a given lower central base $a$ of $G$. Denote by $t(G)$ the product of all finite $\omega_i$  in $\omega$. Clearly, the product $t(G)$, termed the {\em order of lower central torsion} in $G$, does not depend on the base $a$, it is an invariant of $G$. 

To construct a lower central base of $G_X$ of a special type  we make use of the following lemma.
\begin{lemma}\label{le:isolator}
Let $G_X$ be as above. Then for every $i = 1, \ldots,c$ the image of $G_i$ in $\gamma_i(G_X)/\gamma_{i+1}(G_X)$ is a direct summand of $\gamma_i(G_X)/\gamma_{i+1}(G_X)$. 
\end{lemma}
\begin{proof}
By induction on $c$ it suffices to prove the claim for $i = c$. Let $K$ be the normal closure of $N_c(X)$ in $G_X$ and  $B = \gamma_c(G_X) \cap K$. We claim that 
$$\gamma_c(G_X) = G_c \times B.
$$
  To see this consider a commutator $g = [g_1, \ldots,g_c] \in \gamma_c(G_X)$, where $g_i \in G_X$. Notice, that the commutator $\gamma_c = [y_1, \ldots,y_c]$, viewed as a function $ G_X \times \ldots \times G_X \to Z(G_X)$, is multiplicative  in each variable $y_i$. Since $G$ and $N_c(X)$ generate $G_X$ the multiplicativity of the commutator function implies that the subgroups $G_c$ and $B$  generate $\gamma_c(G_X)$. Moreover, the canonical homomorphism   $\phi:G_X \to G \times N_c(X)$ yields 
$\phi(G_c) = G_c \subset  G$ and $\phi(B) \subset N_c(X)$, so $G_c \cap B = 1$  and $\gamma_c(G_X) = G_c \times B$ as required.   \end{proof}

By Lemma \ref{le:isolator} each factor $\gamma_i(G_X)/\gamma_{i+1}(G_X)$ is a direct sum of the type
$$\gamma_i(G_X)/\gamma_{i+1}(G_X) = G_i/G_{i+1} \times B_i.$$
Let, as above, $d_i$ be a base of $G_i/G_{i+1}$ and denote by $b_i$  a base of the abelian group $B_i$.   One can combine the bases  $d_i$ and $b_i$ and arbitrary permute the resulting tuple to get a base of the quotient $\gamma_i(G_X)/\gamma_{i+1}(G_X)$. All the subsequent results of the paper hold for any such a base. However, we order the bases in a particular way to make the writing easier. Namely, we consider  the tuple $e = (b_1d_1)(b_2d_2) \ldots (b_cd_c)$ as  a lower central base of $G_X$ which will be used in the sequent. In this case we say that $b$ {\em extends} the base $a = d_1 \ldots d_c$  of $G$.  
For convenience we choose  tuple $b_1$  of the form   $b_1 = (x_1, \dots,x_k)$. 

Now we fix the  lower central base $e$ of $G_X$ constructed above.  Denote the tail $d_1(d_2b_2) \ldots (d_cb_c)$ of the base $e$ by $f = (f_1, \ldots,f_p)$.

  Then every element $u \in G_X$ can be uniquely written
in the form:

\begin{equation}
\label{eq:normal form}
 u =   x_1^{\gamma _1} \ldots  x_k^{\gamma _k} \  f_1^{\delta_1} \ldots f_p^{\delta_p},
\end{equation}
where $\gamma_j\in \mathbb{Z}_{\omega(x_j)} = \mathbb{Z}, \delta_s \in \mathbb{Z}_{\omega(f_s)}$. Denote $exp(u)  = gcd(\gamma_1, \ldots,\gamma_k)$.  Here we put $exp(u) = 0$  if $\gamma_1 =  \ldots =\gamma_k =0$.

\begin{definition}  The norm $|| \cdot ||$ of an element $u \in G_X$ from (\ref{eq:normal form}) is defined
as

\begin{equation}|| u || = max\{|\gamma _i|,  |\delta _s| \ (i = 1, ..., k;
 s = 1, ... , p)\}.
 \end{equation}
 \end{definition}

The function $u \to ||u||$  gives a ball stratification on $G_X$:  $G_X = \cup_{r\in \N} B_r$, where 
$B_r = \{ u \in G_X: ||u|| \leq r\}.$


\subsection{Satisfiability of random equations}

Let, as usual, $SAT(G, k)\subseteq G_X$ be the set of all equations from $G_X$ which are satisfiable in $G$, and $NSAT(G, k)$ the complement of $SAT(G, k)$ in $G_X$. 
 In  this section we estimate the asymptotic densities of the sets $SAT(G, k)$ and $NSAT(G, k)$
with respect to the ball stratification $G_X = \cup_{r\in \N} B_r$. 

\begin{theorem}
\label{th:4.4}
 Let $G$ be an infinite  finitely generated nilpotent non-abelian group.  Assume that $k\geq 2.$ Then the following estimates hold: 
 \begin{equation}
\label{eq:4.11}
 \liminf \rho_n(SAT(G,k))  \geq \frac{1}{t(G)}\frac{\zeta (k +  h(G))}{\zeta (k)},
 \end{equation}

\noindent where $h(G)$ is the Hirsch length of $G$ and $t(G)$ the order of lower central torsion of $G$;

\begin{equation}
\label{eq:4.11b}
\limsup \rho_n(SAT(G,k)) \leq \frac{\zeta (k + m)}{\zeta (k)},
\end{equation}
where $m$ is the torsion-free rank (Hirsch length) of the abelianization $G/G_2$ of $G$.

\end{theorem}
\begin{proof}
Let $a = d_1d_2\ldots d_c = (a_1, \ldots, a_q)$ be a lower central base of $G$, $e = (x_1, \ldots, x_k,f_1, \ldots,f_p)$ the lower central base of $G_X$ constructed above which extends $a$,  and $\omega = (\omega(x_1), \ldots, \omega(f_p))$ the order of $e$.
We define inductively a set of equations $T \subseteq SAT(G,k)$. To do this we start with a ``generic equation'' of the type (\ref{eq:normal form}) 
with a generic  left-hand side
 $$
 u =   x_1^{\gamma^\ast _1} \ldots  x_k^{\gamma^\ast _k} \  f_1^{\delta^\ast_1} \ldots f_p^{\delta^\ast_p},
$$
where  $ \gamma^\ast_i,  \delta^\ast_s$ are viewed as ``parameters'',  and then inductively specialize the parameters into integers thus obtaining the set $T$ as the result of all such  specializations.  Recall that if the order $\omega_s$ of $f_s$  is finite then the corresponding specialization of $\delta^\ast_s$ takes values  in the set $\Z_{\omega(f_s)} = \{0,1, \ldots, \omega(f_s)-1\}$, while each $\gamma^\ast _i$ takes values in $\Z$ since the elements $x_i$ are of infinite order.

 We  specialize the tuple of parameters $\bar{\gamma}^\ast = (\gamma^\ast_1, \ldots,\gamma^\ast_k)$ into an arbitrary non-zero integer $k$-tuple $\bar{\gamma} = (\gamma_1, \ldots,\gamma_k)$. Let $T_0 = \Z^k \smallsetminus  \{(0, \ldots,0)\}$ be the set of all such  $k$-tuples. 
 Now for each tuple $\bar{\gamma}   \in T_0$ we describe a set $T_{1,\bar{\gamma}}$ of all possible specializations of  the tuple of parameters $\bar{\gamma}^\ast\delta_1^\ast $ (here and further in the proof we use the concatenation notation introduced   in Section \ref{se:Malcev}, so  $\bar{\gamma}^\ast\delta_1^\ast = (\gamma^\ast_1, \ldots,\gamma^\ast_k,\delta_1^\ast)$) that extend the given specialization $\bar{\gamma}^\ast \to \bar{\gamma}. $  
Fix an arbitrary tuple $\bar{\gamma} = (\gamma_1, \ldots,\gamma_k) \in S$. Let $0 \neq \gamma = gcd(\gamma_1, \ldots,\gamma_k)$. Fix $\xi_i  \in \Z \  (i = 1, ... , k)$ such that 
 $\gamma  = \xi _1 \gamma _1 + ... + \xi _k \gamma _k$.  Notice that $f_1$ is an element from the base $a$ of $G$. Specializations of $\delta_1^\ast$  depend on whether $\omega(f_1)$ is  finite or infinite. If $\omega(f_1) = \infty$ then we   specialize $\delta_1^\ast$ into an arbitrary element from  $\gamma\Z$. Otherwise, $\delta_1^\ast$ takes a single value $\lambda_1$, where $\lambda_1 \in \Z_{\omega(f_1)}$ is the remainder of $\gamma$ modulo $\omega(f_1)$. Denote the resulting set of specializations of parameters $\gamma^\ast_1, \ldots,\gamma^\ast_k, \delta_1^\ast$ by $T_{1,\bar{\gamma}}$.
 Put $T_1 = \cup_{\bar{\gamma} \in T_0}  T_{1,\bar{\gamma}} $. For induction we need to prove the following properties of the set of specializations $T_1$.
  
 Fix an arbitrary specialization $\bar{\beta}_1 = (\gamma_1,\ldots, \gamma_k,\delta_1) \in T_{1,\bar{\gamma}}$. We construct a tuple  
 $y^{(\bar{\beta}_1)} = (y_{11}, \ldots y_{1,k})$ of elements from $G$ that solves the ``abridged'' equation 
 $$
u^{(\bar{\beta}_1)}(x_1, \ldots,x_k)  =  x_1^{\gamma_1} \ldots  x_k^{\gamma_k}  f_1^{\delta_1} = 1
 $$
 in the quotient group $G^{(1)} = G/\langle a_2, \ldots,a_q\rangle$.
 Namely, if  $\omega(f_1) = \infty$ then put  
 $$y_{1i} = f_1^{-\delta_1^\prime\xi_i} ,  \ \ i = 1, \ldots,k,$$
where $ \delta_1^\prime = \delta_1/\gamma$. Substituting $y_{1i} $ for $x_i$ in $u^{(\bar{\beta}_1)}(x_1, \ldots,x_k)$ one gets 
$$f_1^{-\delta_1^\prime(\xi_1\gamma_1 + \ldots + \xi_k\gamma_k)}f_1^{\delta_1} = 1.$$
So $y^{(\bar{\beta}_1)}$ is a solution of the abridged equation  $u^{(\bar{\beta}_1)}(x_1, \ldots,x_k) = 1$ in $G^{(1)}$ (even in $G$). In the case when $\omega(f_1)$ is finite put $y_{1i} = f_1^{-\xi_i} ,  \ \ i = 1, \ldots,k.$ Hence
 $$f_1^{-(\xi_1\gamma_1 + \ldots + \xi_k\gamma_k)}f_1^{\lambda_1} = f_1^{-\gamma +\lambda_0} \in 
 \langle a_2, \ldots,a_q\rangle
 $$
 and again $y^{(\bar{\beta}_1)}$ is a solution of the abridged equation   in $G^{(1)}$. This finishes  the basis of induction.

 To describe the induction step of our construction take an arbitrary $i$ such that $1\leq  i <p$.

 Suppose by induction that  a set $T_i$  of $(k+i)$-tuples that give specializations of $\gamma_1^\ast , \ldots, \gamma_k^\ast, \delta_1^\ast, \ldots, \delta_i^\ast$
  is  constructed in such a way that for every tuple $\bar{\beta}_i = (\gamma_1, \ldots,\gamma_k,\delta_1, \ldots,\delta_i) \in T_i$ the following holds.    Let $i^\prime$ be the greatest index such that $i^\prime \leq i$ and $f_{i^\prime}$ belongs to the base $a$ (i.e., $f_{i^\prime} \in G$). Then there is a tuple  of elements $y^{(\bar{\beta}_i)} = (y_{i1}, \ldots, y_{ik})$ of $G$  such that it solves the abridged equation
 $$
 u^{(\bar{\beta}_i)}(x_1, \ldots,x_k)  =  x_1^{\gamma_1} \ldots  x_k^{\gamma_k}  f_1^{\delta_1}\ldots f_{i^\prime}^{\delta_{i^\prime}} = 1
$$
in the quotient group $G^{(i)} = G/A_{i+1}$, where $A_{i+1} = \langle a_{i^\prime+1}, \ldots, a_q\rangle \leq G$.

Now for a given $\bar{\beta}_i = (\gamma_1, \ldots,\gamma_k,\delta_1, \ldots,\delta_i)  \in T_i$ we construct  a set $T_{i+1,\bar{\beta}_i}$  of $(k+i+1)$-tuples that extend the given specialization $\bar{\beta}_i$ by specifying the values of $\delta_{i+1}^\ast$. As before $\gamma = gcd(\gamma_1, \ldots,\gamma_k)$ and $\xi_1, \ldots, \xi_k$ are fixed integers with $\xi_1\gamma_1 + \ldots +\xi_k\gamma_k = \gamma$.  There are three cases to consider.

Case 1. If $f_{i+1}$ is not an element from $G$ (not in the base $a$) then $\delta_{i+1}$ is an arbitrary element from $\Z_{\omega(f_{i+1})}$. In this case 
$$T_{i+1,\bar{\beta}_i} = \{\bar{\beta}_i\delta_{i+1} \mid \delta_{i+1} \in \Z_{\omega(f_{i+1})}\}, \ \ T_{i+1} = \cup_{\bar{\beta}_i \in T_i} T_{i+1,\bar{\beta}_i} .$$
  For any $\bar{\beta}_{i+1} \in T_{i+1}$ put $y^{(\bar{\beta}_{i+1})} = y^{(\bar{\beta}_{i})}$.
Notice, that $(i+1)^\prime = i^\prime$. Since the abridged equations $u^{(\bar{\beta}_i)}(x_1, \ldots,x_k)$ and $u^{(\bar{\beta}_{i+1)})}(x_1, \ldots,x_k) $ are the same, as well as the groups $G^{(i)} = G/\langle a_{i^\prime+1}, \ldots, a_q\rangle$ and $G^{(i+1)} = G/\langle a_{(i+1)^\prime+1}, \ldots, a_q\rangle$,  the tuple $y^{(\bar{\beta}_{i+1})}$ is a solution of the equation $u^{(\bar{\beta}_{i+1)})}(x_1, \ldots,x_k) $  in the group $G^{(i+1)}$, as required.

Case 2. Suppose  $f_{i+1} \in G$ and $\omega(f_{i+1})$ is infinite. In this case $(i+1)^\prime = i+1$. For a given  tuple $\bar{\beta}_i = (\gamma_1, \ldots,\gamma_k,\delta_1, \ldots,\delta_i)  \in T_i$ 
 we need to specialize the parameter $\delta_{i+1}^\ast$   into some $ \delta_{i+1} \in \Z$ in such a way that  the abridged equation  
\begin{equation}\label{eq:special-2}
 u^{(\bar{\beta}_{i+1} )}(x_1, \ldots,x_k)  =  x_1^{\gamma_1} \ldots  x_k^{\gamma_k}  f_1^{\delta_1}\ldots f_i^{\delta_i}f_{i+1}^{\delta_{i+1}} = 1
\end{equation}
has a solution in the group  $G^{(i+1)}$ (here  $\bar{\beta}_{i+1} = \bar{\beta}_i\delta_{i+1}$). We are looking for a solution of the equation (\ref{eq:special-2}) in the form $y^{(\bar{\beta}_{i+1} )} = (y_{i+1,1}, \ldots, y_{i+1,k})$, where $y_{i+1,j} = y_{i,j}f_{i+1}^{t\xi_i}$, where $t \in \Z$ is to be determined later. Notice that substituting $y^{(\bar{\beta}_{i} )}$ into  $u^{(\bar{\beta}_{i+1} )}(x_1, \ldots,x_k)$ results in  an element from the subgroup $A_{i+1} = \langle a_{i+1}, \ldots, a_q\rangle$. Indeed, by induction hypothesis the product 

$$
 u^{(\bar{\beta}_i)}(y^{(\bar{\beta}_{i} )}, \ldots,y^{(\bar{\beta}_{i} )})  =  (y^{(\bar{\beta}_{i} )})^{\gamma_1} \ldots  (y^{(\bar{\beta}_{i} )})^{\gamma_k}  f_1^{\delta_1}(y^{(\bar{\beta}_{i} )})\ldots f_{i^\prime}^{\delta_{i^\prime}} (y^{(\bar{\beta}_{i} )})
$$
lies in $A_{i+1}$. Here  $f_j(y^{(\bar{\beta}_{i} )})$ is the result of substitution of $y_{i,s}$ for $x_s$ ($s= 1, \ldots,k$)  in  $f_j$ ($f_j \in G_X$ are words in the generators of $G$ and $X$). If $i^\prime  < i$ then the elements $f_{i^\prime +1},  f_{i^\prime +2}, \ldots, f_i$ between $f_{i^\prime}$ and $f_{i+1}$ are not in $G$. Suppose $f_{i+1}$ is in the segment  $b_rd_r$ of the base $e$ that comes from the factor   $\gamma_r(G_X)/\gamma_{r+1}(G_X)$, then by construction $f_{i+1} \in d_r$ so if the specifications  $\delta_{i^\prime +1} ^\ast \to \delta_{i^\prime +1},  \delta_{i^\prime +2}^\ast \to \delta_{i^\prime +2}, \ldots, \delta_i^\ast \to \delta_i$  are fixed then the product $f_{i^\prime +1}^{\delta_{i^\prime +1}} f_{i^\prime +2}^{\delta_{i^\prime +2}}, \ldots, f_i^{\delta_{i}}$ is also in the subgroup $A_{i+1}$. It follows that 
$$
 u^{(\bar{\beta}_{i+1} )}(y^{(\bar{\beta}_{i} )}, \ldots,y^{(\bar{\beta}_{i} )})  =  f_{i+1}^{\delta_{i+1} +\mu(y^{(\bar{\beta}_{i})})} (mod \ \langle a_{i+2}, \ldots, a_q\rangle ),
$$
where $\mu(y^{(\bar{\beta}_{i})})$ is an integer that uniquely determined by $y^{(\bar{\beta}_{i})}$.
Here we write that two elements, say $g$ and $h$,  are equal modulo a subgroup $A$ if $g^{-1}h \in A$.
Now substituting $x_i \to y_{i,j}f_{i+1}^{t\xi_i}$  into the abridged equation (\ref{eq:special-2}), i.e,  computing the value $u^{(\bar{\beta}_{i+1} )}(\bar{\beta}_{i+1}, \ldots,\bar{\beta}_{i+1})$, one gets 
$$
f_{i+1}^{t\xi_1\gamma_1  \ldots t\xi_k\gamma_k +\delta_{i+1} +\mu(y^{(\bar{\beta}_{i})})} 
 = f_{i+1}^{t\gamma +\delta_{i+1} +\mu(y^{(\bar{\beta}_{i})})} (mod \ \langle a_{i+2}, \ldots, a_q\rangle ).
$$
Now we define  the specialization $\delta_{i+1}^\ast \to \delta_{i+1}$ by requesting  that $\delta_{i+1} +\mu(y^{(\bar{\beta}_{i})})$ is divisible by $\gamma$.  In this case for such a fixed $\delta_{i+1}$  one can find the value of $t$ such that $y^{(\bar{\beta}_{i+1} )}$ is a solution of the abridged equation $u^{(\bar{\beta}_{i+1} )}(x_1, \ldots,x_k)  =1 (mod \ \langle a_{i+2}, \ldots, a_q\rangle )$, i.e., in the group  $G^{(i+1)}$. Thus, we define 
$$T_{i+1,\bar{\beta}_i} = \{\bar{\beta}_i\delta_{i+1} \mid \gamma \ \text{divides} \ \delta_{i+1} +\mu(y^{(\bar{\beta}_{i})})\}, \ \ T_{i+1} = \cup_{\bar{\beta}_i \in T_i} T_{i+1,\bar{\beta}_i} .$$
and all the inductive conditions hold for $T_{i+1}$.

Case 3. Suppose  $f_{i+1}$ is $G$ and $\omega(f_{i+1})$ is finite. Then, in the notation above, we specialize $\delta_{i+1}$  as such a number in $\Z_{\omega({f_{i+1})}}$ that $\delta_{i+1} +\mu(y^{(\bar{\beta}_{i})}) = \gamma (mod \ \omega(f_{i+1}))$. In this case $y^{(\bar{\beta}_{i+1} )}$ is a solution of the abridged equation $u^{(\bar{\beta}_{i+1} )}(x_1, \ldots,x_k)  =1 (mod \ \langle a_{i+2}, \ldots, a_q\rangle )$ provided  $t = 1$. Again, define 
$$T_{i+1,\bar{\beta}_i} = \{\bar{\beta}_i\delta_{i+1} \}, \ \ T_{i+1} = \cup_{\bar{\beta}_i \in T_i} T_{i+1,\bar{\beta}_i} .$$
It is clear that the inductive hypothesis holds for $T_{i+1}$, as required.

This finishes the construction of the sets of specializations $T_i, i = 1, \ldots,p$. Denote $T = T_p$. 
Each specialization $\beta \in T$ uniquely determine an equation $u^{(\bar{\beta})}$, which by construction has a solution in $G$.  Put $\CE(T) = \{u^{(\bar{\beta})} \mid \bar{\beta} \in T\}$. Then $\CE(T) \subseteq SAT(G,k)$. To prove  (\ref{eq:4.11}),  the first statement of the theorem, it suffices to show that 
$$
\liminf \rho_n(\CE(T)) \geq \frac{1}{t(G)}\frac{\zeta (k +  h(G))}{\zeta (k)}.
$$ 
 Observe, that $||u^{(\bar{\beta})} || = ||\bar{\beta} ||$, so   it suffices to estimate $\rho_r(T)$ in the space $V  = \Z^k \times \Z_{\omega(f_1)} \times \ldots \times \Z_{\omega(f_p)} $ with respect to the ball stratification in the norm $|| \cdot||$. Observe that the ball $B_r(V)$ in $V$ relative to $|| \cdot ||$ is defined by
 $$
 B_r(V) = \{(\gamma_1, \ldots, \gamma_k, \delta_1, \dots,\delta_p) \mid  \gamma_i \in \Z,  |\gamma_i| \leq r,  \delta_s \in \Z_{\omega(f_s)}, |\delta_s| \leq r\}.
 $$
Below we give an estimate of $\rho_r(T)$ for a sufficiently large $r$. Fix an arbitrary $r$ such that $r > \omega(f_i)$ for each finite $\omega(f_i)$,  $i = 1, \ldots,p$.  

Notice that the construction of the sets $T_i$  depends crucially on $\gamma = gcd(\gamma_1, \ldots,\gamma_k)$ where $\bar{\gamma}^\ast \to \bar{\gamma}$ is the initial specialization of $\bar{\gamma}^\ast$ into $T_0$.
As in Theorem \ref{th:abelian-two-var} it is  convenient to partition each set $T_i$ into its $\gamma$-slices $T_i^\gamma$   which are defined inductively as follows. 
For every $\gamma \neq 0$ put 
 $$
 T_0^\gamma = \{(\gamma_1, \ldots,\gamma_k)  \in T_0 \mid \gamma = gcd(\gamma_1, \ldots,\gamma_k)\},
 $$
  so $T_0 = \cup_{\gamma \neq 0}T_0^\gamma$.
Suppose that $T_i^\gamma$ is defined for some $i <p$. Then put    
 $$T_{i+1}^\gamma = \cup_{\bar{\beta}_i \in T_i^\gamma}  T_{i+1,,\bar{\beta}_i}, $$
 so  
 $$
 T_{i+1} = \cup_{\gamma \neq 0}T_{i+1}^\gamma.
 $$

One has  
\begin{equation}\label{eq:Tunion}
\rho_r(T) = \frac{ |T \cap B_r(V)|}{|B_r(V)|}  = \frac{\sum_{\gamma = 1}^r |T^\gamma \cap B_r(V)|}{|B_r(V)|} =  \sum_{\gamma = 1}^r\rho_r(T^\gamma).
\end{equation}
Now we estimate $\rho_r(T^\gamma)$ for a fixed $0\neq \gamma \leq r$. Following the construction of $T_{i+1}^\gamma$, we  fix a specialization $\bar{\beta} \in T_i ^\gamma$ with $||\bar{\beta} ||\leq r$ and estimate the number of extensions $\bar{\beta} \delta_{i+1} \in T_{i+1,\bar{\beta} }$ with $||\delta_{i+1}|| \leq r$. More precisely, let $\pi(T_{i+1,\bar{\beta} }) \subseteq \Z_{\omega(f_{i+1})}$ be the projection of $T_{i+1,\bar{\beta} }$ on  its last coordinate. We  estimate   the frequency 
$$\rho_r^{\Z_{\omega(f_{i+1})}}(\pi(T_{i+1,\bar{\beta} })) = \frac{|\pi(T_{i+1,\bar{\beta} }) \cap B_r(\Z_{\omega(f_{i+1})})|}{|B_r(\Z_{\omega(f_{i+1})})|}$$
of $\rho_r(\pi(T_{i+1,\bar{\beta} }))$ in the space $\Z_{\omega(f_{i+1})}$ with respect to $||\cdot||$.
There are three cases to consider.

Case 1. $f_{i+1} \not \in G$. Then $\delta_{i+1}$ is an arbitrary element from $\Z_{\omega(f_{i+1})}$
In this case 
\begin{equation} \label{eq:Case1}
\rho_r^{\Z_{\omega(f_{i+1})} }(\pi(T_{i+1,\bar{\beta} }) ) = 1.
\end{equation}

Case 2. $f_{i+1} \in G$ and $\omega(f_{i+1})$ is infinite.  Then
$$T_{i+1,\bar{\beta}_i} = \{\bar{\beta}_i\delta_{i+1} \mid \gamma \ \text{divides} \ \delta_{i+1} +\mu(y^{(\bar{\beta}_{i})})\},$$
where $\gamma, \mu(y^{(\bar{\beta}_{i})}) \in \Z$ are fixed. It is easy to see that in this case 
\begin{equation}\label{eq:Case2}
\rho_r^{\Z_{\omega(f_{i+1})}}(\pi(T_{i+1,\bar{\beta} })) = \frac{|\gamma\Z \cap B_r(\Z)|}{|B_r(\Z)|}) = \rho_r^\Z(\gamma \Z)
\end{equation}
(here $\rho_r^\Z(\gamma \Z)$ is the frequency of $\gamma\Z$ in $\Z$) and one can use the estimates from Proposition \ref{pr:gamma-Z-k}. In particular, for every $r \geq \gamma$ (which is the case here)
$$|\rho_r^\Z(\gamma \Z) - \frac{1}{\gamma}| \leq \frac{4}{r}. $$

Case 3. $f_{i+1}$ is $G$ and $\omega(f_{i+1})$ is finite. In this case $\delta_{i+1}^\ast$ has a single specialization into $\Z_{\omega(f_{i+1})}$. So in this case 
\begin{equation}\label{eq:Case3}
\rho_r^{\Z_{\omega(f_{i+1})}}(\pi(T_{i+1,\bar{\beta} })) = \frac{1}{|\Z_{\omega(f_{i+1})}|} = \frac{1}{\omega(f_{i+1})}.
\end{equation}

Now we estimate the  frequency $\rho_r^{V_{i+1}}(T_{i+1}^\gamma)$ of $T_{i+1}^\gamma$  in the space $V_{i+1} = \Z^k \times \Z_{\omega(f_1)} \times \ldots \times \Z_{\omega(f_{i+1})} $ with respect to the ball stratification in the norm $|| \cdot||$. Since   $T_{i+1}^\gamma = \cup_{\bar{\beta}_i \in T_i^\gamma} T_{i+1,\bar{\beta}_i} $ it follows that 
\begin{equation} \label{eq:rhoTi}
\rho_r^{V_{i+1} }(T_{i+1}^\gamma) = \frac{|T_{i+1}^\gamma \cap B_r(V_{i+1} )|}{|B_r(V_{i+1} )|} = \rho_r^{V_i}(T_i^\gamma)\rho_r^{\Z_{\omega(f_{i+1})} }(\pi(T_{i+1,\bar{\beta} })).
\end{equation}
Combining (\ref{eq:rhoTi}) with (\ref{eq:Case1}),  (\ref{eq:Case2}),  (\ref{eq:Case3}) one gets
\begin{equation}\label{eq:prob-1}
\rho_r(T^\gamma) = \rho_r^{\Z^k}(T_0^\gamma)\rho_r^{\Z_{\omega(f_1)} }(\pi(T_1^\gamma)) \ldots \rho_r^{\Z_{\omega(f_p)} }(\pi(T_p^\gamma)) 
\end{equation}

We can clarify and simplify the product (\ref{eq:prob-1}) as follows.  Firstly, remove all the factors $\rho_r^{\Z_{\omega(f_i)} }(\pi(T_i))$  equal to 1 in (\ref{eq:prob-1}). Secondly, collect  all the factors 
$\rho_r^{\Z_{\omega(f_i)} }(\pi(T_i)) = \frac{1}{\omega(f_{i})}$ with  finite $\omega(f_{i})$ together, so  the resulting product is equal to $\frac{1}{t(G)}$. Observe, that all the other factors $\rho_r^{\Z_{\omega(f_i)} }(\pi(T_i))$ in (\ref{eq:prob-1}) are those where $\omega(f_i) = \infty$, in which case all   of them are equal to  $\rho_r^\Z(\gamma \Z)$. The number of such factors is equal to $h(G)$ - the Hirsh length of $G$. Notice also that $T_0^\gamma$ is precisely the set $P_{k,\gamma}$ of all non-trivial $\gamma$-primitive elements in $\Z^k$, so $\rho_r^{\Z^k}(T_0^\gamma) = \rho_r^{\Z^k}(P_{k,\gamma})$. This implies that
\begin{equation}\label{eq:Tgamma}
\rho_r(T^\gamma) = \frac{1}{t(G)}\rho_r^{\Z^k}(P_{k,\gamma})\left (\rho_r^\Z(\gamma \Z) \right )^{h(G)},
\end{equation}
From (\ref{eq:Tunion}) and (\ref{eq:Tgamma}) we get 
$$
\rho_r(T) =   \sum_{\gamma = 1}^r\rho_r(T^\gamma) = \frac{1}{t(G)}\sum_{\gamma = 1}^r \rho_r^{\Z^k}(P_{k,\gamma})\left (\rho_r^\Z(\gamma \Z) \right )^{h(G)}.
$$
The sum on the right is precisely the sum (\ref{eq:SAT(A,k)})  estimated in Theorem \ref{th:abelian-two-var}. It was shown there that 
$$
\lim_{r\to \infty}\sum_{\gamma = 1}^r \rho_r^{\Z^k}(P_{k,\gamma})\left (\rho_r^\Z(\gamma \Z) \right )^{h(G)} = \frac{\zeta (k +  h(G))}{\zeta (k)}.
$$
Hence
$$
\lim_{r \to \infty}\rho_r(T) = \frac{1}{t(G)} \frac{\zeta (k +  h(G))}{\zeta (k)}.
$$
Therefore
$$
\liminf \rho_n(\CE(T)) \geq \frac{1}{t(G)}\frac{\zeta (k +  h(G))}{\zeta (k)},
$$
which proves the first statement of the theorem.  
 
  To prove the second inequality (\ref{eq:4.11b}) consider the maximal  free abelian quotient $\bar{G}= \Z^m$ of $G$. It is easy to see that $\bar{G}  = G / Is(\gamma_2G) $, where $Is(\gamma_2G)$ is the isolator of $\gamma_2G$ in $G$, in other words $\bar{G}$ can be obtained from the abelianization $G/\gamma_2G$ by factoring out the torsion subgroup of $G/\gamma_2G$. It follows that the  rank $m$ is 
  precisely the torsion-free rank (Hirsch length) of $G/\gamma_2G$.
  
  Obviously, if an equation 
  $$
  u =   x_1^{\gamma _1} \ldots  x_k^{\gamma _k} \  f_1^{\delta_1} \ldots f_p^{\delta_p} = 1
  $$
of the type (\ref{eq:normal form})  does not have a solution in $\bar{G}$ then it has no solution in $G$. In the quotient $\bar{G}$  the equation $u = 1$ takes the form 
 $$
 x_1^{\gamma _1} \ldots  x_k^{\gamma _k} \  a_1^{\delta_1} \ldots a_m^{\delta_m} = 1
 $$
 where $d_1 = (a_1, \ldots,a_n)$ is the base corresponding to  $G / \gamma_2G$ and $a_1, \ldots,a_m$ are all the elements in $d_1$ with $\omega(a_i) = \infty$.
 
 By Theorem \ref{th:abelian-two-var} for the free abelian group $\bar{G}$ one has 
 $$
 \rho(SAT(\bar{G},k)) = \frac{\zeta (k + m)}{\zeta (k)}.
 $$
 Therefore,
 \begin{equation}
\label{eq:4.15}
 \liminf_{r \to \infty}\rho_r ( NSAT(G,k)) \geq 1 - \frac{\zeta (k + m)}{\zeta (k)}.
 \end{equation}

It follows that 
$$\liminf_{r \to \infty} \rho_r(NSAT(G,k))= \liminf_{r \to \infty} (1 - \rho_r(SAT(G,k))) = 1-\limsup_{r \to \infty} \rho_n(SAT(G,k)).
$$
Hence
$$
\frac{\zeta (k + m)}{\zeta (k)} \geq \limsup_{r \to \infty} \rho_r(SAT(G,k)),
$$
as required.
\end{proof}

\begin{theorem}
\label{th:4.6}
 Let $G$ be an infinite finitely generated nilpotent group. Then the set $SAT(G, 1)$ is negligible. 

\end{theorem}

\begin{proof} Let $G/G_2 \simeq \Z^m \times K$, where $K$ is a finite group. Notice that $m \geq 1$ since $G$ is infinite.  If an equation $u \in G_X$ has a solution in $G$ then it has a solution in the quotient $\Z^m$. Now the result follows from Theorem \ref{th:abelian-one-var} (by an argument similar to the one at the end of the proof of Theorem \ref{th:4.4}).

\end{proof}

\end{document}